\documentclass[10pt]{article}
\usepackage{amsmath}
\usepackage{amssymb}
\usepackage{amsthm}
\usepackage{exscale}
\usepackage{epsfig}
\usepackage[mathscr]{eucal}
\usepackage{float}
\usepackage{bm}
\usepackage{bbm}

\newtheorem{theorem}{Theorem}[section]
\newtheorem{lemma}[theorem]{Lemma}
\newtheorem{corollary}[theorem]{Corollary}

\newtheorem{remark}[theorem]{Remark}

\numberwithin{equation}{section}

\numberwithin{equation}{section}

\renewcommand{\phi}{\varphi}

\begin{document}

\title{Substituting Independent Processes}
\author{Manfred Denker}


\maketitle

\begin{abstract}
It is shown by constructing Rohlin's canonical measures that for a strictly stationary, $d$-dimensional vector-valued process $\mathcal X=(X_n)_{n\in\mathbb N}$  there exists another strictly stationary $d$-dimensional process $\mathcal U=(U_n)_{n\in \mathbb N}$ with uniform one-dimensional marginals and with the same mixing properties as $\mathcal X$, such that $\mathcal X$ is a finitary factor of $\mathcal U$ of coding length $1$, and such that the projection map is order preserving in each coordinate. As an application this extends the a.s. approximation of the empirical distribution function of weakly dependent random vectors with continuous distribution function in \cite{147.1} and \cite{147.3} to the general case.
\end{abstract}

\section{Introduction}\label{sec:147.1}

Let $X_1$, $X_2$, $X_3 \ ...$ be $d$-dimensional random vectors $X_n=(X_n^{(1)},...,X_n^{(d)})$, which form a strictly stationary process. Denote by $F$ the distribution function of $X_1$and by $F_\nu$ that of  $X_1^{(\nu)}$  ($\nu=1,...,d$). For several problems in probability and statistics it is necessary (or at least simplifying the situation) to represent the process  $\mathcal X=(X_n)_{n\ge 1}$  as a factor of some other $d$-dimensional process $\mathcal U=(U_n)_{n\ge 1}$, $U_n=(U_n^{(1)},...,U_n^{(d)})$, with uniformly distributed $U_n^{(\nu)}$, $\nu=1,...,d$  (cf. \cite{147.3}, \cite{147.4}, \cite{147.5}). Moreover, the process $\mathcal U$ should have the same dependence properties  as $\mathcal X$, such as ergodicity, weak, strong and uniform  mixing. It is crucial in these applications to have a factor map $\pi:([0,1]^d)^{\mathbb N}\to (\mathbb R^d)^{\mathbb N}$ which factorizes coordinate-wise and each coordinate is mapped by a (lexicographical) order-preserving correspondence. Such a factor map is given by $\pi=(\pi_i)_{i\in\mathbb N} $, $\pi_i=(F_1^{-1},...,F_d^{-1})$.  

If the random variables $X_n$ are independent and if $d=1$, then $\pi=(F^{-1},F^{-1},$
\newline$....)$ and independent, uniformly distributed $U_n$ ($n\in\mathbb N$) factorize via $\pi$ to $\mathcal X$, for $F^{-1}U_n$ has the same distributions as $X_n$. However, if $d>1$, the situation is more complicated (cf.\cite{147.4}) and if, moreover, $\mathcal X$ is no longer an independent process, some additional problems arise. As long as each $F_\nu$ has no jumps, $\pi_i$ is invertible  ($i\in \mathbb N$) and so is $\pi$, hence $\pi$ transports the measure $\nu$ on $(\mathbb R^d)^{\mathbb N}$, given by $\mathcal X$, to some measure $\mu$ on $([0,1]^d)^{\mathbb N}$ possessing all properties of interest. When $F_\nu$ has point masses, this  construction is no longer possible and in this note we present  an adequate construction. Note that it is fairly easy using the theory of Lebesgue spaces to obtain some factor map from some process with uniform distributed marginals and the same mixing properties, but such a construction leaves the factor map unspecified.

The present construction  of the $\mathcal U$-process bears the following idea: Suppose for simplicity that $d=1$, that the process $\mathcal X$ is ergodic and that its distribution function $F$ has only one jump at $a\in \mathbb R$. We may think of  $\mathcal X$ to be represented  as the canonical model on $\mathbb R^{\mathbb N}$ by a shift-invariant probability measure $\nu$, $T:\mathbb R^{\mathbb N}\to \mathbb R^{\mathbb N}$ denoting the shift transformation. Pick a generic point $x=x(n)$ for $\nu$, which also has the correct frequency of occurrences of $a$ (cf. \cite{147.2}). Now let $y\in \pi^{-1}(x)$, $\pi=(F^{-1},F^{-1},...)$. Then for all $f\in \pi C_b(\mathbb R^{\mathbb N})\subset C_b([0,1]^{\mathbb N})$, $\lim \frac 1n\sum_{0\le k<n} f(T^k(y))$ exists, i.e., every weak limit of $\frac 1n\sum_{0\le k<n}\epsilon_{T^k(y)}$ for $y\in \pi^{-1}(x)$  projects to $\nu$. In order to choose $y$ properly if $x(n)=a$, define $y(n)$ according to a generic point for independent, uniformly on $\{s: F^{-1}(s)=a\}$ distributed random variables. Such a $y$ is generic and defines a probability with nice properties.

 As an application of this construction one obtains an a.s. invariance principle for the empirical distribution function  of strongly mixing  processes $\mathcal X$. This extends the  corresponding result of Berkes, Philipp \cite{147.1} and Philipp, Pinzur \cite{147.3}, because in these theorems the continuity of $F$ is assumed (see the next section). It is still an open problem to prove the almost sure invariance principles for functions of mixing random variables in general (cf. \cite{147.1}). The present method does not seem to cover such an extension.
 
    \section{The main result and an application} \label{sec:147.2}  

Instead of considering the abstract process $\mathcal X=(X_n)_{n\ge 1}$  we consider its natural extension on $(\mathbb R^d)^\mathbb Z$ given by the shift $T$-invariant measure $\nu$ on the Borel field $\mathcal B$ on $(\mathbb R^d)^\mathbb Z$. Note that $(\pi_t)_{t\ge 1}$ has the same distribution as $\mathcal X$, where $\pi_t$ ($t\in \mathbb Z$) is given by the projection
$$ \pi_t((x(n))_{n\in \mathbb Z})=x(t)\qquad \left(x=(x(n))_{n\in \mathbb Z}\in (\mathbb R^d)^\mathbb Z\right).$$

$\mathcal X= (X_n)$ is called ergodic, weakly mixing or mixing if the dynamical system $((\mathbb R^d)^\mathbb Z, \nu)$ has this property (cf.~\cite{147.2}) We also recall the following definitions of mixing properties of the process $(X_n)$: $\mathcal X$ is called
\begin{description}
\item[(a)] uniformly mixing with rate $\varphi(n)$ if
$$\varphi(n):=\sup\{|\nu(B|A)-\nu(B)|\ : A\in \sigma(p_t: t\le 0);\ B\in \sigma(p_t: t\ge n)\}\to 0,$$
\item[(b)] absolutely regular with rate $\beta(n)$ if
$$\beta(n):=\sup\{|\nu(B|\sigma(p_t:t\le 0))-\nu(B)|\ : \ B\in \sigma(p_t: t\ge n)\}\to 0,$$
\item[(c)] strongly mixing with rate $\alpha(n)$ if $\lim_{n\to\infty}\alpha(n)=0$ where
$$\alpha(n):=\sup\{|\nu(A\cap B)-\nu(A)\nu(B)|\ : A\in \sigma(p_t: t\le 0);\ B\in \sigma(p_t: t\ge n).$$
\end{description}
Here $\sigma(\mathcal F)$ denotes the $\sigma$-field generated by all Borel measurable functions in $\mathcal F$.  Note that the mixing rates for $\mathcal X$ in the usual definition do not change when passing to its natural extension.

If $F$ denotes a distribution function on $\mathbb R$, we denote by
$$ F^{-1}(s)=\inf\{ t:\ F(t)\ge s\}$$ 
its inverse. In the above situation, when $X_1$ has the distribution $F$ on $\mathbb R^d$ and $X_1^{(k)}$ has the distribution function $F^{(k)}$ ($k=1,...,d$), let $F^{-1}:=((F^{(1)})^{-1},...,$
\newline$(F^{(d)})^{-1})$, then $\pi:=(...,\pi_{-1},\pi_0,\pi_1,...)$ with $\pi_i=F^{-1}$ defines a measurable map $\pi: (]0,1[^d)^\mathbb Z\to (\mathbb R^d)^\mathbb Z$. We shall say that the process $\mathcal X$ is a factor  (via $\pi$) of the process $(Y_m)$, which can be represented on $(]0,1[^d)^\mathbb Z$ by the stationary measure $\mu$ if $\pi$ transports $\mu$ into $\nu$, i.e.~ $\pi\mu=\nu$.

We shall prove the following theorem in this note:
\begin{theorem}
Let $(X_n)_{n\in \mathbb Z}$ be a $d$-dimensional, strictly stationary process with distribution function $F$ for $X_1$ (resp. $F^{(k)}$ for $X_1^{(k)}$,$k=1,...,d$). Then there exists a $]0,1[^d$-valued, strictly stationary process $(U_m)_{m\in \mathbb Z}$ which factorizes via $\pi=(...,F^{-1},F^{-1},F^{-1},...)$ onto $(X_n)_{n\in \mathbb Z}$. The one-dimensional marginals of $U_m$ ($m\in \mathbb Z$) are uniformly distributed. Moreover, if the process $(X_n)_{n\in \mathbb Z}$ is ergodic, weakly mixing, mixing, strongly mixing with rate $\alpha(n)$, absolutely regular with rate $\beta(n)$, or uniformly mixing with rate $\varphi(n)$, then the process $(U_m)_{m\in \mathbb Z}$ can be constructed with the same properties, including the rates.
\end{theorem}

Before proving the result, we give an application to a.s. approximations of empirical distribution functions. If $u=(u_1,...,u_d)$, $v=(v_1,...,v_d)\in \mathbb R^d$ we write $u\le v$ if $u_i\le v_i$ for all $i=1,...,d$.
Let $F_n(t)=\frac 1n\sum_{i=1}^n (\mathbb I_{\{X_i\le t\}}-F(t))$ ($t\in \mathbb R^d$) denote the (centered) empirical distribution function of $X_1,...,X_n$. A separable Gaussian process $K(s,t)$ on $\mathbb R^d\times [0,\infty[$ is called a Kifer process if $K(s,0)=0$, $\lim_{s_1,...,s_d\to \infty} K(s,t)=0$, $\lim_{s_i\to-\infty}K(s,t)=0$, $EK(s,t)=0$ and $EK(s,t)K(s',t')=\min\{t,t'\} \Gamma(s,s')$, for all $s,s'\in \mathbb R^d$, $t,t'\ge 0$. Here $\Gamma$ denotes the covariance function of a separable Gaussian process $G$ on $\mathbb R^d$ with $\lim_{s_1,...,s_d\to\infty} G(s)=0$ and $\lim_{s_i\to-\infty} G(s)=0$. Note that in the strongly mixing case with rates satisfying $\sum \alpha(n)^{\delta/(2+\delta)}<\infty$ for some $\delta>0$, the series
\begin{equation}\label{eq1}
\Gamma(s,s'):= Eg_1(s)g_1(s')+\sum_{n\ge 2} (Eg_1(s)g_n(s')+Eg_1(s')g_n(s))
\end{equation}
converges absolutely, where $g_n(s)=\mathbb I_{\{X_n\le s\}}-F(s)$ ($s\in \mathbb R^d$). $\Gamma$ is a positive definite, symmetric function. From \cite{147.3} we recall a theorem of  Berkes, Philipp, Pinzur:
\begin{theorem}[\cite{147.1, 147.3}]\label{theo:147.1}
For a strongly mixing, strictly stationary, $d$-dimensional process $(X_n)_{n\ge 1}$ with mixing rates $\alpha(n)=O(n^{-4-d(1+\epsilon)})$ for some $0<\epsilon\le \frac 14$ and with continuous distribution function $F$, one can redefine the empirical process $R(s,t)= [t] F_{[t]}(s)$ of $(X_n)$ on a richer probability space on which there exists a Kifer process $K$ with covariance function $\min\{t,t'\} \Gamma(s,s')$, where $\Gamma$ is defined by (\ref{eq1}), such that for some $\lambda>0$ and all $\tau>0$
$$ \sup_{0\le t\le \tau}\sup_{s\in \mathbb R^d} |R(s,t)-K(s,t)|= O(\sqrt{\tau}(\log \tau)^{-\lambda})\quad a.s. $$
\end{theorem} 
A slightly stronger result in the case $d=1$ has been obtained earlier in \cite{147.1} by Berkes and Philipp. The proof of the above theorem in Philipp-Pinzur (\cite{147.3}) is based on its proof in the case of uniform one-dimensional marginals and the main theorem in this note in the case of a continuous distribution function $F$. Hence we obtain immediately
\begin{corollary}\label{cor:147.1}
The Berkes-Philipp-Pinzur theorem remains true if the hypothesis of continuity of $F$ is dropped.
\end{corollary}
As it is discussed in \cite{147.1} one can immediately derive some other corollaries from the last corollary, concerning the functional central limit theorem and the law of iterated logarithm. For a formulation and its discussion of such theorems we refer to \cite{147.1}.

\section{The proof of the theorem}\label{sec:147.3}
We shall use the notation introduced so far.

Denote by $a_i^{(k)}$ ($k=1,...,d; i=1,2,...$) the points in $\mathbb R$ where $F^{(k)}$ has a jump, and put $A_i^{(k)}= \left((F^{(k)})^{-1}\right)^{-1}\{a_i^{(k)}\}\subset ]0,1[$.  The pre-image with respect to $F^{-1}$ of the partition of $\mathbb R^d$ into points is a partition of $]0,1[^d$, denoted by $\gamma$. It consists of  sets $C_1\times...\times C_d$ where $C_k$ is either a point or some $A_i^{(k)}$ ($k=1,...,d$). Equally we may consider $\gamma$ as a partition on $(]0,1[^d)^{\mathbb Z}$ by  partitioning the $0$-coordinate and then we can apply the shift transformation $S$ on $(]0,1[^d)^\mathbb Z$, $S(x(n)_{n\in \mathbb Z})=(y(n))_{n\in \mathbb Z}$, $y_{n+1}= x_n$, to $\gamma$, obtaining $S^n\gamma $ ($n\in\mathbb Z$). The common refinement of all partitions $S^n\gamma$ ($n\in \mathbb Z$) is denoted by $\Sigma(\gamma)$. Then, for $\nu$-a.e.~$\omega\in (\mathbb R^d)^\mathbb Z$  there exists an atom $A(\omega)\in \Sigma(\gamma)$ such that $\pi(y)=\omega$ for all $y\in A(\omega)$. Since $\Sigma(\gamma)$ is algebra-isomorphic (mod $\nu$) to the  Borel field on $(\mathbb R^d)^\mathbb Z$, $\nu$ can be considered  as a measure on $\Sigma(\gamma)$.

In order to prove the theorem, we shall define canonical measures $\mu_\omega$ on the sets $A(\omega)$ in an appropriate way, obtaining $\mu$ by integrating with respect to $\nu$. For $s\in \mathbb R$ and $k=1,...,d$ defines $\mu(s,k)$ on $]0,1[$ by
$$ \mu(s,k)=\begin{cases}
\epsilon_{(F^{(k)})^{-1}(s)}\qquad & \mbox{\rm if $F^{(k)}$ has no jump at $s$}\\
\lambda(\cdot)|A_i^{(k)}) & \mbox{\rm if $s=a_i^{(k)}$ for some $i$}.
\end{cases}
$$
(Here $\epsilon_z$ denotes the unit mass in $z$ and $\lambda$ the Lebesgue measure.)

Let $A(\omega)\in \Sigma(\gamma)$, $\omega=(\omega_n)_{n\in \mathbb Z}$,  $\Omega_n=(\omega_n^{(1)},...,\omega_n^{(d)})$ be given. Define
$$ \mu_\omega= \prod_{n\in \mathbb Z} \prod_{k=1}^d \mu(\omega_n^{(k)},k),$$
to be the product measure on $A(\omega)$ whose $n$-th coordinate marginal is given by the product measure $\mu(\omega_n^{(1)},1)\times...\times \mu(\omega_n^{(d)},d)$. If $D\subset ]0,1[^d$ we shall consider it as well as a subset of  $(]0,1[^d)^\mathbb Z$  by $\{y=(y(n)\in (]0,1[^d)^\mathbb Z: \ y(0)\in D\}$.

\begin{lemma}\label{lem:147.1}
For every bounded,  measurable function $f$ on $(]0,1[^d)^\mathbb Z$ the map $\omega\to \int f d\mu_\omega$ is $\mathcal B$-measurable.
\end{lemma}
\begin{proof}
Let $D_j=D_j^{(1)}\times...\times D_j^{(d)}$ ($j=-n,...,n$) where each $D_j^{(k)}$ is contained in some $A_i^{(k)}$ or in $]0,1[\setminus\bigcup_i A_i^{(k)}$. Set  $D=\bigcap_{j=-n}^n S^jD_j$. Then, if $A(\omega)\cap D=\emptyset$, $\int \mathbbm 1_D d\mu_\omega=0$, and if $A(\omega)\cap D\ne \emptyset$, then for all $D_j^{(k)}\subset ]0,1[\setminus \bigcup_i A_i^{(k)}$ we have  $\omega_j^{(k)}\in D_j^{(k)}$ and for all $D_j^{(k)}\subset A_i^{(k)}$ we have  $\omega_j^{(k)}\in A_i^{(k)}$. It follows that for all $\omega\not\in \pi D$ \ $\int \mathbbm 1_D d\mu_\omega=0$ and for all $\omega\in \pi D$
$$ \int \mathbbm 1_D d\mu_\omega= \prod_k\frac{\lambda(D_j^{(k)}}{\lambda(A_i^{(k)}}$$
where the product extends over all $j,k$ such that $D_j^{(k)}\subset A_i^{(k)}$ for some $i$. Since $\pi D$ is $\mathcal B$ measurable, we have the measurability of $\omega\to \int \mathbbm 1_D d\mu_\omega$. Consequently, every bounded, measurable function has this property.
\end{proof}

Let us now define $\mu$. We set for a bounded, measurable function $f$ on $(]0,1[^d)^\mathbb Z$
\begin{equation}\label{eq:147.2}
\int f d\mu =\int\int f d\mu_\omega d\nu(\omega).
\end{equation}
By Lemma \ref{lem:147.1} the right hand side of (\ref{eq:147.2}) is well defined.

\begin{lemma}\label{lem:147.2}
The relation (\ref{eq:147.2}) defines a stationary probability measure $\mu$ on $(]0,1[^d)^\mathbb Z$, which - by $\pi$ - is transported to $\nu$ and which has uniformly distributed one-dimensional marginals.
\end{lemma}

\begin{proof}
First consider (\ref{eq:147.2}) for bounded, continuous functions $f$ on $(]0,1[^d)^\mathbb Z$. Clearly, the right hand side of (\ref{eq:147.2}) is linear in $f$ and if $f_n\downarrow 0$ pointwise, then $\int f_n d\mu_\omega\downarrow 0$ pointwise, hence by the dominated convergence theorem $\int f_n d\mu\to 0$. Hence by the Riesz representation theorem, (\ref{eq:147.2}) defines a measure, which clearly is normalized.

In order to see that $\pi\mu=\nu$ it suffices to show that $\nu$ and $\mu$ are identical on $\Sigma(\alpha)$. Let $D=\bigcap_{j=-n}^n T^jD_j\in T^{-n}\alpha\wedge...\wedge T^n\alpha, \ D_j= D_j^{(1)}\times...\times D_j^{(d)}$. Then we may assume $D_j^{(k)} \subset ]0,1[\setminus \bigcup_i A_i^{(k)}$ or $D_j^{(k)}= A_i^{(k)}$ for some $i$. Hence by (\ref{eq:147.2}) and the argument in the proof of Lemma \ref{lem:147.1}
$$ \mu(D)= \int \mathbbm 1_D d\mu=\int \mathbbm 1_{\pi D} d\nu= \nu(\pi(D)).$$

The uniform distribution is calculated as follows. Let 
$$D= ]0,1[\times...\times ]0,1[\times D^{(k)}\times ]0,1[\times...\times]0,1[\quad k=1,...,d.$$
If $D^{(k)}\subset ]0,1[\setminus \bigcup_i A_i^{(k)}$, then by (\ref{eq:147.2})
\begin{eqnarray*}
\mu(D)&=& \nu(\pi D)= \nu(p_0\in F^{-1}(D)\\
&=& \mu(p_0^{(k)}\in (F^{(k)})^{-1} (D^{(k)})\\
&=& P(X_0^{(k)}\in (F^{(k)})^{-1}(D^{(k)})=\lambda(D^{(k)}).
\end{eqnarray*}
If $D^{(k)}\subset \bigcup_i A_i^{(k)}$ for some $i$, then by (\ref{eq:147.2})
$$ \mu(D)= \int_{\pi  D} \frac{\lambda(D^{(k)}}{\lambda(A_i^{(k)}} d\nu =
\frac{\lambda(D^{(k)}}{\lambda(a_i^{(k)}} P(X_0^{(k)}=a_i^{(k)})= \lambda (D^{(k)}).$$
\end{proof}

\begin{lemma}\label{lem:147.3}
If $\nu$ is ergodic then $\mu$ is also ergodic. 
\end{lemma}

\begin{proof}
We have to show that
$$ \lim_{n\to\infty} \frac 1n \sum_{0\le k<n} \int f\cdot(g\circ T^k) d\mu = \int fd\mu\int gd\mu$$
for every pair $(f,g)$ of bounded measurable functions.

Note that the functions $\omega\to\int fd\mu_\omega$ and $\omega\to \int g d\mu_\omega$ are bounded and measurable (Lemma \ref{lem:147.1}). Moreover, by construction, ($h$ measurable, bounded)
$$ \int hd\mu_\omega= \int h\circ T d\mu_{T(\omega)}$$
where $T$ denotes the shift (to the right) on $(\mathbb R^d)^\mathbb Z$. It follows now from (\ref{eq:147.2}) that for $\sigma(\pi_t:\ -N\le t\le N)$ measurable functions
\begin{eqnarray*}
&& \lim \frac 1n \sum\int f\cdot (g\circ T^k) d\mu = \lim \frac 1n \sum \int\int f\cdot(g\circ T^k) d\mu_\omega d\nu(\omega)\\
&& =\lim \frac 1n \sum\int\left[\int fd\mu_\omega\int g\circ T^k d\mu_\omega\right]d\nu(\omega)\\
&&= \lim \frac 1n \sum \int\left[\int fd\mu_\omega\int g d\mu_{T(\omega)}\right]d\nu(\omega)\\
&& = \int\int fd\mu_\omega d\nu(\omega)\cdot \int\int g d\mu_\omega  d\nu(\omega)= \int f d\mu \int gd\mu,
\end{eqnarray*}
since $\nu$ is assumed to be ergodic.
\end{proof}

\begin{lemma}\label{lem:147.4}
If $\nu$ is weakly mixing, then $\mu$ is also weakly mixing.
\end{lemma}

\begin{proof}
We have to show that
$$ \lim_{n\to\infty} \frac 1n \sum_{0\le k<n} \left|\int f\cdot(g\circ T^k) d\mu - \int fd\mu\int gd\mu\right|=0$$
for all bounded, measurable functions $f$ and $g$.
\begin{eqnarray*}
&& \lim \frac 1n \sum\left|\int f\cdot(g\circ T^k) d\mu -\int fd\mu \int g d\mu\right|\\
&& = \lim \frac 1n \sum\left|\int f d\mu_\omega\int g\circ T^k d\mu_\omega \nu(d\omega)-\iint fd\mu_\omega d\nu(\omega)\iint g d\mu_\omega d\nu(\omega)\right|\\
&&=0,
\end{eqnarray*}
since $\nu$ is weakly mixing.
\end{proof}

\begin{lemma}\label{lem:147.5}
If $\nu$ is mixing, then $\mu$ is also mixing.
\end{lemma}

\begin{proof}
As in the foregoing lemmas show that
$$ \int f\cdot(g\circ T^k) d\mu \to \int fd\mu \int g d\mu.$$
\end{proof}

\begin{lemma}\label{lem:147.6}
If $\nu$ is strongly mixing, then $\mu$ is strongly mixing with the same rates $\alpha(n)$.
\end{lemma}

\begin{proof}
Let $\alpha(n):= \sup\{|\nu(A\cap B)-\nu(A)\nu(B)|: \ A\in \sigma(\pi_t:\, t\le 0); \ B\in \sigma(\pi_t:\, t\ge n)\}$. We have to show the same relation for the measure $\mu$. It suffices to prove it for sets $A$ and $B$ belonging to  dense subclasses in the $\sigma$-algebras under consideration. These subclasses are described by finite disjoint unions  of sets $D=\bigcap_{j=r}^S T^jD_j$, $D_j=D_j^{(1)}\times ...\times D_j^{(d)}$, for which either $D_j^{(k)}\subset A_i^{(k)}$ for some $i$ or $D_j^{(k)}\subset ]0,1[\setminus \bigcup A_i^{(k)}$.  For a set $D_j=D_j^{(1)}\times...\times D_j^{(d)}$ as just described, let $\tilde D_j= \tilde D_j^{(1)}\times...\times \tilde D_j^{(d)}$ where
$\tilde D_j^{(k)}= A_i^{(k)}$ iff $D_j^{(k)}\subset A_i^{(k)}$ for some $i$ and where $\tilde D_j^{(k)}= D_j^{(k)}$ otherwise. Note that  $\pi^{-1}\pi(D_j)= \tilde D_j$ and $\pi^{-1}\pi(D)= \tilde D=\bigcap_{j=r}^s T^j\tilde D_j$.

 Let $A=\bigcup_{l=1}^L D(l)$, $D(l) =\bigcap_{j=-m}^0 T^jD_j(l)$ and $B=\bigcup_{i=n}^I C(i)$, $C(i)=\bigcap_{j=n}^{n+M} T^jC_j(i)$ belong to the dense subclasses, and assume that the $D(l)$ (resp. $C(i)$) are disjoint.
 
 The corresponding $\tilde D(l)$  (resp. $\tilde C(i)$) may not be disjoint; however, it is not difficult to check that, making the $\tilde D(l)$ (resp. $\tilde C(i)$) disjoint, we can represent $A$ (resp. $B$) by a finite union of sets $\tilde D(l) \cap G(l)$ (resp. $\tilde C(i)\cap H(i)$) where $G(l)$ is a finite union of sets of the form $\bigcap_{j=-m}^0 T^j G_j$ with $G_j^{(k)}=]0,1[$ iff 
 $D_j^{(k)}(l)\subset ]0,1[\setminus\bigcup A_s^{(k)}$ and $G_j^{(k)}\subset A_s^{(k)}$ iff $D_j^{(k)}(l)\subset A_s^{(k)}$ for some $s$ (and analogously for $H(i)$.
 
 With these decompositions of $A$ and $B$ in mind we have
 \begin{eqnarray*}
 &&\mu(A\cap B)-\mu(A)\mu(B)\\
 &=& \sum_l \mu(\tilde D(l)\cap G(l)\cap B)-\mu(\tilde D(l)\cap G(l))\mu(B)\\
 &=& \sum_l \int_{\tilde D(l)}\int \mathbb I_B \mathbb I_{G(l)} d\mu_\omega  d\nu(\omega) - \int_{\tilde D(l)}\int \mathbb I_{G(l)} d\mu_\omega d\nu(\omega)\cdot\nu(B)\\
 &=&\sum_l \mu_\omega(G(l))\left( \int_{\tilde D(l)}\int_B d\mu_\omega-\int_{\tilde D(l)}d\nu \mu(B)\right)\\
 &\le & \sum_l^* \mu(\tilde D(l)\cap B)- \mu(\tilde D(l))\mu(B)
 \end{eqnarray*}
 \begin{eqnarray*}
 &=&   \mu(\sum_l^*\tilde D(l)\cap B)- \mu(\sum_l^*\tilde D(l))\mu(B)\\
 &\le& \sum_i^*\mu(\sum_l^*\tilde D(l)\cap \tilde C(i))- \mu(\sum_l^*\tilde D(l))\mu(\tilde C(i))\\
 &=& \sum_i^*\nu(\pi\left(\sum_l^*\tilde D(l)\right)\cap \pi(\tilde C(i)))- \nu(\pi\left(\sum_l^*\tilde D(l)\right))\nu(\pi(\tilde C(i)))\\
 &\le& \alpha(n)
 \end{eqnarray*}
since on $\tilde D(l)$ $\int_{G(l)} d\mu_\omega$ is constant. Here $\sum^*$ denotes summation over positive terms.

 The converse inequality is shown analogously. One really has equality in the estimate when taking the supremum over $A$ and $B$, since this follows from the fact that $\mu$ restricted to $\Sigma(\alpha) $ coincides with $\nu$. 
\end{proof}

\begin{lemma}\label{lem:147.7}
Let the measure $\nu$ be uniformly mixing. Then $\mu$ is also uniformly mixing  with the same rate $\varphi(n)$.
\end{lemma}

\begin{proof}
Let
$$\varphi(n)=\sup\{|\nu(B|A)-\nu(B)|: \ A\in\sigma(\pi_t:\, t\le 0);\ B\in \sigma(\pi_t:\, t\ge n\}.$$
We shall show this relation for $\mu$ over the same dense subclasses as in Lemma \ref{lem:147.6}.

Let $A$, $B$ be given as in the proof of Lemma \ref{lem:147.6}. Then
\begin{eqnarray*}
&& \mu(A\cap B)-\mu(A)\mu(B)\\
&\le& \sum^*_k \mu(A\cap \tilde C(k))-\mu(A)\mu(\tilde C(k)) \\
&=& \mu(A\cap B^*)-\mu(A)\mu(B^*)\qquad(B^*=\sum_k^* \tilde C(k))\\
&=& \sum_l\mu_\omega(G(l)) \left[\mu(\tilde D(l)\cap B^*)- \mu(\tilde D(l))\mu(B^*)\right]\\
&\le& \varphi(n) \sum_l^* \mu_\omega(G(l))\mu(\tilde D(l))\\
&=&\le \varphi(n)\mu(A).
\end{eqnarray*}
The lower bound is shown analogously.
\end{proof}

\begin{lemma}\label{lem:147.8}
Let the measure $\nu$ be absolutely regular. Then $\mu$ is absolutely regular  with the same rates.
\end{lemma}

\begin{proof}
Let
$$ \beta(n):= E\sup \{ |\nu(B)|\sigma(\pi_t:\, t\le 0)\}-\nu(B)|: \ B\in \sigma(\pi_t:\, t\ge n\}.$$

Let $\mathcal F=\{D(1),...,D(L)\}$ be a finite $\sigma$-algebra in the past $\sigma$-algebra $\mathcal B_0$ of the $\mu$-process. Assume that
$D(l)=\bigcap_{j=-m}^0 T^j D_j(l)$, where the $D_j(l)$ are as in Lemmas \ref{lem:147.6} and \ref{lem:147.7}. It follows from the proof of Lemma \ref{lem:147.7}, that for a set $B$ as before
\begin{eqnarray*}
|\mu(B|D(j))-\mu(B)|&\le &  |\nu(\pi(B^*) 
|\pi\tilde D(j))- \nu(\pi(B^*))|\\
&=& |\mu(B^*|\tilde D(j)) -\mu(B^*)|. 
\end{eqnarray*}
(We may assume that the sets $\tilde D(j)$ are pairwise disjoint or equal.)
Hence
\begin{eqnarray*}
|\mu(B|\mathcal F)-\mu(B)|&\le& |\mu(B^*|\tilde D(j))-\mu(B^*)|\\
&=& |\mu(B^*|\tilde{\mathcal F})- \mu(B^*)|\qquad \mbox{on}\ \tilde D(j),
\end{eqnarray*}
where $\tilde {\mathcal F}=\{\tilde D(1),...,\tilde D(L)\}$.

Therefore
$$ |\mu(B|\mathcal F)-\mu(B)|\le \sup_{U\in \sigma(\pi_t: t\ge n)} |\mu(\pi^{-1}(U)|\tilde{\mathcal F})-\mu(\pi^{-1}(U))|$$
and
$$ \int \sup_B|\mu(B|\mathcal F)-\mu(B)| d\mu\le \int \sup_{B\in \sigma(\pi_t:t\ge n)}|\nu(B|\pi\tilde{\mathcal F})- \nu(B)| d\nu$$
where  we have used that $\nu(B|\pi\tilde{\mathcal F})\circ \pi= \mu(\pi^{-1}B|\tilde{\mathcal F})$.

If $\mathcal F\subset \mathcal F'$, then $E_\mu(\mu(B|\mathcal F')-\mu(B)|\mathcal F)= \mu(B|\mathcal F)-\mu(B)$  so that $\sup|\mu(B|\mathcal F)-\mu(B)|$ is a supermartingale, hence
$$ \int \sup_B|\mu(B)|\mathcal F')-\mu(B) d\mu\le \int\sup_{B\in \sigma(\pi_t:t\ge n)}|\nu(B|\pi\tilde {\mathcal F})-\nu(B)| d\nu$$
for all $\mathcal F'\supset \mathcal F$. By the submartingale convergence theorem
$$\int \sup_B|\mu(B|\mathcal B_0)-\mu(B)| d\mu \le \int \sup_B|\nu(B|\pi\tilde{\mathcal F})-\nu(B) d\nu$$
for every $\tilde{\mathcal F}$. Since $\pi\tilde{\mathcal F}$ runs through a generating family of finite $\sigma$-algebras of $\sigma(\pi: t\le 0)$, the past of the $\nu$-process, the same arguments apply to show that
$$ \int \sup_B|\mu(B|\mathcal B_0)-\mu(B)|\le \beta(n).$$
Note that equality in the last estimate follows as before.
\end{proof}

\begin{remark}\label{rem:147.1}
If $A$, $B\subset (]0,1[^d)^\mathbb Z$ are measurable, $A$ with respect to the coordinates $\le 0$ and $B$ with respect to the coordinates $\ge n$, note that $\omega\to \int \mathbbm 1_A d\mu_\omega$ is $\sigma(\pi_t: t\le 0)$ measurable and $\omega\to \int \mathbbm 1_B d\mu_\omega$ is $\sigma(\pi_t:t\ge n)$ measurable. If $\nu$ is strongly mixing it follows from the well-known basic inequality for this mixing condition that
\begin{eqnarray*}
&& |\mu(A\cap B)-\mu(A)\mu(B)|\\
&& \qquad = \left|\int \mu_\omega(A\cap B)d\nu(\omega) -\int \mu_\omega(A)d\nu(\omega) \int \mu_\omega(B) d\nu(\omega)\right|\\
&&\qquad = \left|\int \mu_\omega(A)\mu_\omega(B)d\nu(\omega) -\int \mu_\omega(A)d\nu(\omega) \int \mu_\omega(B) d\nu(\omega)\right|\\
&&\qquad \le 4\alpha(n).
\end{eqnarray*}
This is a slightly weaker statement than Lemma \ref{lem:147.6}.

 Similarly, there is a weaker statement (and easier proof) for Lemma \ref{lem:147.7} (consequently also for Lemma \ref{lem:147.8}). 
\end{remark}

\section{Appendix}\label{sec:147.4} The following Theorem\footnote{The appendix is  the conference abstract: Journ\'ees de Th\'eorie Ergodique, C.I.R.M.--Luminy, Marseille, 5--10 Juillet 1982}   has been proven by 
Berkes, Philipp (\cite{147.1}) and Philipp, Pinzur (\cite{147.3}):
\begin{theorem}
For a strongly mixing, strictly stationary sequence $(X_n)_{n\in \mathbb Z}$ of real valued random variables with rates $\alpha(n)=O(n^{-5+\epsilon})$ ($0<\epsilon<\frac 14$), and with continuous distribution function $F$, one can redefine the empirical process 
$$ R(s,t)= \sum_{i=1}^{[t]} \mathbbm 1_{]-\infty,s]}(X_i)-F(s)\quad s\in \mathbb R, t\ge 0$$
on a richer probability space on which there exists a Kifer process $K$ with covariance structure $\min(t,t') \Gamma(s,s')$ such that for some $\lambda>0$
$$ \sup_{0\le t\le T} \sup_{s\in\mathbb R} |R(s,t)-K(s,t)|=O(\sqrt{T} (\log T)^{-\lambda}\quad a.s.$$
\begin{eqnarray*}
\Gamma(s,s') = &&  E(\mathbbm 1_{\{X_1\le s\}}-F(s))(\mathbbm 1_{\{X_1\le s'\}}-F(s'))\\
&&\quad+\sum_{n\ge 2} E(\mathbbm 1_{\{X_1\le s\}}-F(s))(\mathbbm 1_{\{X_n\le s'\}}-F(s'))\\
&&\quad + \sum_{n\ge 2}E(\mathbbm 1_{\{X_1\le s'\}}-F(s'))(\mathbbm 1_{\{X_n\le s\}}-F(s)).
\end{eqnarray*}
\end{theorem}
In order to drop the condition of continuity of $F$ one considers the following: Let $F^{-1}(s)=\inf\{t: F(t)\ge s\}$, $R_F= \mbox{\rm Im}\ F^{-1}$, and $\Pi:]0,1]^{\mathbb Z}\to \mathbb R^{\mathbb Z}$ be defined by 
$$ \Pi((y(n))_{n\in\mathbb Z})= (F^{-1}(y(n)))_{n\in \mathbb Z}.$$
\begin{theorem} Let $\nu$ be a shift invariant probability on $R_F^{\mathbb Z}$.Then there exists a shift invariant, ergodic probability $\mu$ on $]0,1[^{\mathbb Z}$ with the following properties:
\begin{enumerate}
\item $\mu\{ y(n)_{n\in \mathbb Z}:\ y(0)\le t\} =t$\quad($0<t<1$)
\item $\Pi\mu=\nu$
\item If $\nu$ is weakly mixing (mixing, strongly mixing with rate $\alpha(n)$) then $\mu$ has the same property.
\end{enumerate}
\end{theorem}
Added in 2018: The result has been cited several times and is here published as a complete manuscript:
\begin{enumerate}
\item  Bradley, Richard C.:
On possible mixing rates for some strong mixing conditions for N-tuplewise independent random fields. 
Houston J. Math. 38 (2012), no. 3, 815--832.
\item  Bradley, Richard C.:
Introduction to strong mixing conditions. Vol. 1--3. Kendrick Press, Heber City, UT, 2007.
\end{enumerate}

\end{document}